\documentclass{amsart}
\usepackage{fullpage}
\usepackage{fullwidth}
\usepackage{amsmath,amssymb}
\usepackage{thm-restate}
\usepackage{mathtools}
\usepackage{enumerate}
\usepackage{mathrsfs}
\usepackage[noend]{algpseudocode}
\usepackage{algorithm,algorithmicx}
\usepackage{hyperref}

\newcommand{\e}{\text{e}}
\newcommand{\Z}{\mathbb{Z}}
\newcommand{\cl}{\mbox{cl}}
\newcommand{\comp}{\mbox{comp}}
\newcommand{\prob}[2]{\mathbb{P}_{#1}\left(#2\right)}

\newcommand{\entropy}[1]{\mathscr{H}\!\!\left(#1\right)}

\newcommand{\defeq}{\coloneqq}

\let\olduparrow\uparrow
\renewcommand{\uparrow}{\mathord{\olduparrow}}

\newtheorem{lemma}{Lemma}
\newtheorem{theorem}[lemma]{Theorem}
\newtheorem{corollary}[lemma]{Corollary}
\newtheorem{conjecture}[lemma]{Conjecture}
\newcommand{\ignore}[1]{}
\newcommand{\dual}{^*}

\algnewcommand\AlgInput{\State\textbf{Input: }}
\algnewcommand\AlgOutput{\State\textbf{Output: }}

\begin{document}


\title{Enumerating matroids of fixed rank}
\author{Rudi Pendavingh}
\author{Jorn van der Pol}
\address{Eindhoven University of Technology, Eindhoven, the Netherlands}
\email{\{R.A.Pendavingh,J.G.v.d.Pol\}@tue.nl}
\thanks{This research was supported by the Netherlands Organisation for Scientific Research (NWO) grant 613.001.211.}
\maketitle

\begin{abstract}
It has been conjectured that asymptotically almost all matroids are sparse paving, i.e.\ that~$s(n) \sim m(n)$, where~$m(n)$ denotes the number of matroids on a fixed groundset of size~$n$, and~$s(n)$ the number of sparse paving matroids.
In an earlier paper, we showed that~$\log s(n) \sim \log m(n)$. The bounds that we used for that result were dominated by matroids of rank~$r\approx n/2$.
In this paper we consider the relation between the number of sparse paving matroids~$s(n,r)$ and the number of matroids~$m(n,r)$ on a fixed groundset of size~$n$ of fixed rank~$r$. In particular, we show that~$\log s(n,r) \sim \log m(n,r)$ whenever~$r\ge 3$, by giving asymptotically matching upper and lower bounds.

Our upper bound on~$m(n,r)$ relies heavily on the theory of matroid erections as developed by Crapo and Knuth, which we use to encode any matroid as a stack of paving matroids. Our best result is obtained by relating to this stack of paving matroids an antichain that completely determines the matroid.

We also obtain that the collection of essential flats and their ranks gives a concise description of matroids.
\end{abstract}

\section{Introduction}


In their enumeration of matroids on up to eight elements, Blackburn, Crapo, and Higgs~\cite{BlackburnCrapoHiggs1973} observed that a large fraction of matroids are paving, i.e.\ the cardinality of each circuit is at least the rank of the matroid. Crapo and Rota~\cite[p.\ 3.17]{CrapoRotaBook} speculated that this observed behaviour holds in general, i.e.\ that `paving matroids will predominate in any enumeration of matroids'.
Mayhew, Newman, Welsh, and Whittle~\cite[Conjecture~1.6]{MayhewNewmanWelshWhittle2011} formally conjectured that this is true in a strong sense, namely that asymptotically almost all matroids are paving. They noted that their conjecture is equivalent to the (seemingly stronger) conjecture that asymptotically almost all matroids are sparse paving, i.e.\ that~$s(n) \sim m(n)$.
Here, $m(n)$ denotes the number of matroids on a fixed groundset of size~$n$, and~$s(n)$ denotes the number of those matroids that are sparse paving (a matroid is called sparse paving if both it and its dual are paving, or equivalently, a matroid of rank~$r$ is sparse paving if each $r$-set of elements is either a basis or a circuit-hyperplane).

In a recent paper~\cite{PvdP2015B}, we proved a weaker version of the conjecture, namely that~$\log s(n) \sim \log m(n)$\protect\footnote{Throughout the paper, we use~$\log$ to denote the logarithm with base 2, and we wil use~$\ln$ for the natural logarithm.}. The bounds on~$\log s(n)$ and~$\log m(n)$ that we obtained are dominated by matroids of rank~$r \approx n/2$. In fact, for~$r$ sufficiently close to~$n/2$, we have~$\log s(n,r) \sim \log m(n,r)$, where we write~$m(n,r)$, resp.\ $s(n,r)$, for the number of matroids, resp.\ sparse paving matroids, of rank~$r$ on a fixed groundset of size~$n$.

When we tried to apply the same methods to matroids of small rank, we obtained much weaker results. In~\cite[Section~4.3]{PvdP2015B} we speculated that even for matroids of small rank the estimate~$\log s(n,r) \sim \log m(n,r)$ holds. The main contribution of this paper is that this is indeed the case for constant~$r \ge 3$.

Our methods rely heavily on the theory of truncations and erections of matroids, which was developed by Crapo~\cite{Crapo1970} and Knuth~\cite{Knuth1975}. If~$M$ is a matroid of rank~$r$, then its truncation is the matroid~$T(M)$ on the same groundset, which has as independent sets all independent sets of~$M$ of cardinality at most~$r-1$. A matroid~$N$ is called an erection of~$M$ if~$M = T(N)$ or~$M=N$. In general, a matroid can have more than one erection; in particular, the uniform matroid~$U(r,n)$ has all paving matroids of rank~$r+1$ as erections.

Every matroid can be seen as its truncation, augmented by information describing its set of hyperplanes. We show that the information required to erect the original matroid from its truncation encodes a paving matroid (showing, in passing, that the uniform matroid has the largest number of erections among all matroids with the same groundset and rank). Inductively, we can describe every matroid as a stack of paving matroids, each encoding the next erection.

A flat in a matroid is called essential, roughly, if ``its existence does not follow from what the flats of lesser rank are'' (Higgs, in~\cite{Crapo1970}). The essential flats of a matroid, together with their ranks, completely specify the matroid. Each paving matroid in our stack encodes the essential flats of the corresponding rank. Higgs suggested that essential flats may offer a concise description of matroids. We show that this is indeed the case, by proving an upper bound on the number of essential flats. For~$r \approx n/2$, this results in an upper bound on~$\log m(n,r)$ that asymptotically matches the best known lower bound.

Encoding a matroid as a stack of paving matroids is closely related to work by Knuth~\cite{Knuth1975}, who described a procedure for generating matroids by constructing a sequence of erections (although he did not use his term). Knuth noted that this procedure could be used to generate random matroids, by randomising the choice for each erection, and he asked whether his procedure could be analysed carefully enough to provide good bounds on the number of matroids. Our analysis shows that this is the case, at least for matroids of fixed rank.





Bounds obtained in earlier papers focussed on the case~$r \approx n/2$, and aimed at describing the sets of bases and nonbases of the matroid. This type of argument makes it difficult to take into account possible ``deep dependencies'', although such deep dependencies may have a profound effect on the number of matroids.
In the current paper, we describe a matroid as a stack of paving matroids, and exploit the relation between those paving matroids. Our analysis shows that a matroid~$M$ of rank~$r$ on~$n$ elements can be described by a structured set~$\mathcal{V}(M)$ of at most~$\frac{1}{n-r+1}\binom{n}{r}$ of its circuits.

\subsection{Results}

In this paper, we study the relation between~$s(n,r)$, $p(n,r)$ and~$m(n,r)$ when the rank~$r$ is fixed. (Here, $p(n,r)$ refers to the number of paving matroids.)

This is easy for matroids of rank~1. There are~$2^n-1$ matroids of rank~1. Each of these matroids is paving, and only~$n+1$ of those are sparse as well (these are the matroids with at most a single loop).

The case for matroids of rank~2 is relatively easy as well, as matroids, paving matroids, and sparse paving matroids of rank~2 each correspond to other combinatorial objects of which very precise asymptotics are known.
\begin{restatable}{theorem}{thmrank}
	\label{thm:rank2}
	$2 \log s(n,2) \sim \log p(n,2) \sim \log m(n,2)$ as~$n\to\infty$.
\end{restatable}

In~\cite[Section~4.3]{PvdP2015B}, we argued that a recent result by Keevash~\cite{Keevash2015} on the number of Steiner systems gives a lower bound for sparse paving matroids: as each Steiner system~$S(r-1,r,n)$ gives the set of hyperplanes of a sparse paving matroids, for fixed~$r$ we obtain
\begin{equation}\label{eq:intro_bound_keevash}
	\log s(n,r) \ge \frac{1}{n-r+1} \binom{n}{r} \log(\e^{1-r} (n-r+1)(1+o(1)))\qquad\text{as $n\to\infty$},
\end{equation}
at least when~$n$ and~$r$ satisfy certain natural divisibility conditions. We show that~\eqref{eq:intro_bound_keevash} holds even when these conditions do not hold.

Using the entropy method, we prove an upper bound on the number of sparse paving matroid that asymptotically matches the lower bound of~\eqref{eq:intro_bound_keevash}.
\begin{restatable}{theorem}{thmsnrentropy}\label{thm:snr_entropy}
	For all~$r \le n$, 
	$\log s(n,r) \le \frac{1}{n-r+1} \binom{n}{r} \log (n-r+2)$.
\end{restatable}


As pointed out in the introduction, every matroid can be seen as a matroid of lower rank (its truncation), augmented with extra information that encodes a paving matroid.
\begin{restatable}{theorem}{thmfixrank}\label{thm:fixrank}
	For all fixed~$r$, $\log m(n,r) \le \left(1 + \frac{r+o(1)}{n-r+1}\right) \log p(n,r)$ as~$n\to\infty$.
\end{restatable}

A matroid is paving precisely when its truncation is a uniform matroid. If the paving matroid is of rank~$r$, we show that its hyperplanes can be identified by a small collection of $r$-sets. As such a small collection can be chosen in a limited number of ways, the following bound on paving matroids follows.
\begin{restatable}{theorem}{thmpupperbound}\label{thm:p_upperbound}
	For all~$3\le r\le n$, we have~$\log p(n,r) \le \frac{1}{n-r+1} \binom{n}{r} \log \left(\e(n-r+1)\right)$.
\end{restatable}
It follows that for all fixed~$r \ge 3$,
\begin{equation}\label{eq:fixed_rank_equivalence}
	\log s(n,r) \sim \log p(n,r) \sim \log m(n,r) \qquad\text{as $n\to\infty$.}
\end{equation}

Theorem~\ref{thm:p_upperbound} turns out to be a special case of a more general result. Recall that it was obtained by writing a paving matroid as a uniform matroid plus a small collection of sets describing the hyperplanes of that paving matroid. A general matroid can be described inductively as a sequence of erections~$(\mathcal{V}_0, \mathcal{V}_1, \ldots, \mathcal{V}_{r-1})$, in which~$\mathcal{V}_k$ is a set of $(k+1)$-sets that (roughly) encodes the flats of rank~$k$. Choosing the sets~$\mathcal{V}_i$ carefully, we can ensure that their union forms an antichain, and that the same is true for a related set system~$\mathcal{A}(M)$.
An application of the well-known LYM-inequality to~$\mathcal{A}(M)$ translates to an inequality concerning a weighted sum of the cardinalities of the~$\mathcal{V}_k$. This shows that~$|\mathcal{V}(M)|$ is small, and hence bounds the number of matroids.
For matroids of rank~3, we can work directly with the detailed inequality to obtain a slightly stronger bound on~$m(n,3)$, which we can then use to obtain upper bounds for general rank.
\begin{restatable}{theorem}{thmmnrbound}\label{thm:mnr_bound}
	For all~$r\ge 3$ and~$n\ge r+12$, $\log m(n,r) \le \frac{1}{n-r+1}\binom{n}{r} \log(\e(n-r+1))$.
\end{restatable}

Comparing the upper bounds on~$\log s(n,r)$, $\log p(n,r)$, and~$\log m(n,r)$, we find that they only differ in the constant factor inside the logarithm: For sparse paving matroids this is~$1$, while for paving and general matroids the constant factor is~$\e$.
This observation suggests that if our goal is to show that~$\frac{s(n,r)}{m(n,r)}\not\to 1$ as~$n\to\infty$, it might be easiest to find a gap between~$s(n,r)$ and~$p(n,r)$.

Starting from a sparse paving matroid, we obtain a paving matroid by extending one of its circuit-hyperplanes by an extra point (and possibly deleting some other circuit-hyperplanes). For rank~$r=3$, this construction gives a class of paving matroids that is large compared to the class of sparse paving matroids.
\begin{restatable}{theorem}{thmpvss}\label{thm:pn3-vs-sn3}
	$\liminf\limits_{n\to\infty} \frac{p(n,3)}{s(n,3)} > 1$.
\end{restatable}

The following theorem gives an upper bound on the number of essential flats of a matroid. 

\begin{restatable}{theorem}{thmessentialflatsbound}\label{thm:essential-flats-bound}
	Suppose that~$r\ge 3$ and~$n\ge 2r$. A matroid~$M$ of rank~$r$ on~$n$ elements has at most~$\frac{1}{n-r+1}\binom{n}{r}$ essential flats.
\end{restatable}

\subsection{Remainder of the paper}

In Section~\ref{sec:preliminaries}, we state some preliminaries. We also prove that~\eqref{eq:intro_bound_keevash} holds for general rank~$r$, as well as a matching upper bound. In Section~\ref{sec:truncation_erection}, we review the relation between the results of Crapo and Knuth on erections. In Section~\ref{sec:erection_paving}, we explore the relation between paving matroids and erections, and prove Theorem~\ref{thm:fixrank} and Theorem~\ref{thm:p_upperbound}.

In Section~\ref{sec:antichains}, we define two antichains associated with matroids, and show how these antichains can be exploited to prove Theorem~\ref{thm:mnr_bound}. In this section, we will also prove Theorem~\ref{thm:essential-flats-bound}.

Constructions for paving matroids are considered in Section~\ref{sec:constructions}, and it is in this section that Theorem~\ref{thm:pn3-vs-sn3} is proved.

Finally, in Section~\ref{sec:closing}, we address a number of issues that are still open.

\section{Preliminaries}\label{sec:preliminaries}

\subsection{Matroids}

A {\em matroid} is a pair~$(E,\mathcal{I})$ consisting of a finite groundset~$E$, and a collection~$\mathcal{I}$ of subsets of~$E$ called independent sets, that satisfies the following axioms: (i) $\emptyset \in \mathcal{I}$, (ii) if~$I \subseteq I'$ and~$I' \in \mathcal{I}$, then~$I \in \mathcal{I}$, (iii) if $I, I' \in \mathcal{I}$ and $|I'| > |I|$, then there is~$x \in I'\setminus I$ such that $I\cup\{x\} \in \mathcal{I}$.

We assume familiarity with this definition, as well as with the definitions of circuit, basis, rank, flat etc., for which we refer to~\cite{OxleyBook}.

A matroid of rank~$d+1$ is called {\em paving} if each of its circuits has cardinality at least~$d+1$. This is the case if and only if its hyperplanes form a set system known as a {\em $d$-partition}. A family of subsets of~$E$ is a $d$-partition if it contains at least two sets, each of which has cardinality at least~$d$, and every~$d$-set of the groundset is contained in exactly one member.

A matroid is called {\em sparse paving} if both it and its dual are paving. Equivalently, a matroid of rank~$r$ is sparse paving if and only if every $r$-set is either a basis or a circuit-hyperplane.

We will use~$m(n,r)$, $p(n,r)$, and $s(n,r)$, respectively, for the number of matroids, paving matroids, and sparse paving matroids of rank~$r$ on groundset~$E = [n]$.

Sparse paving matroids on~$E$ of rank~$r$ are related to stable sets\protect\footnote{A {\em stable set} in a graph~$G$ is a set of vertices in~$G$, no two of which are adjacent. Stable sets are also called independent sets, but as this term has a different meaning in graph theory, we avoid its use.} of the {\em Johnson graph}~$J(n,r)$. Given a finite set~$E$, the Johnson graph~$J(E,r)$ is the graph on vertex set~$\binom{E}{r}$, in which any pair of vertices is connected by an edge if and only if the two vertices intersect in exactly~$r-1$ elements. We use~$J(n,r)$ as a shortcut for~$J([n],r)$. The following lemma exhibits the relation between the Johnson graph and sparse paving matroids. It was essentially proved by Piff and Welsh~\cite{PiffWelsh1971}; a proof can be found in~\cite[Lemma~8]{BPvdP2015}.
\begin{lemma}
	Let~$\mathcal{B} \subseteq \binom{E}{r}$. Then~$\mathcal{B}$ is the collection of bases of a sparse paving matroid if and only if~$\binom{E}{r} \setminus \mathcal{B}$ is a stable set in~$J(E,r)$.
\end{lemma}

\subsection{A lower bound on the number of sparse paving matroids of fixed rank}

For integers~$t<k<n$, a {\em Steiner system}~$S(t,k,n)$ is a collection of $k$-subsets of~$[n]$ (called blocks) with the property that each $t$-subset of~$[n]$ is contained in a unique block. Every~$S(r-1,r,n)$ forms the collection of hyperplanes of a sparse paving matroid of rank~$r$. So, Steiner systems~$S(r-1,r,n)$ form a subclass of sparse paving matroids of rank~$r$.

In order for an~$S(t,k,n)$ to exist, the parameters~$t$, $k$ and~$n$ should satisfy certain natural divisibility conditions. Recently, Keevash~\cite[Theorem~6.1]{Keevash2015} provided a good estimate for the number of~$S(t,k,n)$, provided these divisibility conditions hold and~$n$ is sufficiently large. Let us write~$D(t,k,n)$ for the number of~$S(t,k,n)$.

\begin{lemma}[{\cite[Theorem~6.1]{Keevash2015}}]\label{lemma:keevash}
	For any~$t$, $k$, there is~$n_0$ such that if~$n > n_0$ and~$\binom{k-i}{t-i} \mid \binom{n-i}{t-i}$ for all~$0 \le i \le t-1$, then
	\begin{equation*}
		D(t,k,n) \ge \left(\e^{1-Q} N + o(N)\right)^{Q^{-1} \binom{n}{t}},
	\end{equation*}
	where~$N = \binom{n-t}{k-t}$ and~$Q = \binom{k}{t}$.
\end{lemma}
Specialising the above lemma to~$k = r$, $t = r-1$, we find that for fixed~$r$
\begin{equation}\label{eq:sp_lowerbound_steiner}
	\log s(n,r) \ge \frac{1}{n-r+1} \binom{n}{r} \log\left(\e^{1-r} (n-r+1)(1+o(1))\right),
\end{equation}
provided
\begin{equation}\label{eq:divisibility_conditions}
	(r-i) \mid \binom{n-i}{r-i-1}\qquad\text{for all $0 \le i \le r-2$,}
\end{equation}
and~$n$ is sufficiently large.

We claim that~\eqref{eq:sp_lowerbound_steiner} holds for sufficiently large~$n$, even if the divisibility conditions are violated. This would not be unreasonable: even though violation of the divisibility conditions forbids the existence of~$S(r-1,r,n)$, according the Erd\H{o}s-Hanani conjecture~\cite{Rodl1985} there exist partial Steiner systems that cover all but a vanishing fraction of $(r-1)$-subsets. Starting from a partial Steiner system, the set of hyperplanes of a sparse paving matroid is obtained by adding to it all $(r-1)$-subsets that are not contained in one of the blocks.

Similar estimates of the number of such ``near-optimal'' partial Steiner systems should hold. In fact, the same line of reasoning that Keevash uses in his proof of Lemma~\ref{lemma:keevash}, can be used to prove the following result.

\begin{theorem}\label{thm:general_keevash}
	For all fixed~$r$, \eqref{eq:sp_lowerbound_steiner} holds as~$n\to\infty$.
\end{theorem}

We sketch a proof of the theorem, which is a special case of Keevash's proof of Lemma~\ref{lemma:keevash}. For this, it will be useful to consider the matchings in a certain hypergraph.

Let~$\mathcal{F}_r$ be the hypergraph on vertex set~$\binom{[n]}{r-1}$, in which the edges consist of sets of the form~$\binom{A}{r-1}$, for~$A \in\binom{[n]}{r}$. The hypergraph~$\mathcal{F}_r$ has~$N = \binom{n}{r-1}$ vertices, is $r$-uniform, and is~$D$-regular with~$D = n-r+1$. Moreover, the codegree of any pair of distinct vertices is at most~1.

A matching in~$\mathcal{F}_r$ consists of a set of edges, no two of which share a vertex. Hence, matchings in~$\mathcal{F}_r$ are in 1-1 correspondence with partial~$S(r-1,r,n)$.

We consider the random greedy matching process on~$\mathcal{F}_r$, i.e.\ we start with the empty matching, and in each iteration of the process we choose an edge uniformly at random from the remaining edges in the hypergraph, add it to the matching, and delete it and any incident edges from~$\mathcal{F}_r$. The process stops when no further edges can be added to the matching.

Let the random variable~$M$ be the number of edges that are selected into the matching. The random greedy matching process is studied, among others, by Bennett and Bohman~\cite{BennettBohman2012}, where it is shown that with high probability (i.e.\ with probability tending to 1 as~$n\to\infty$)
\begin{equation*}
	M\ge\left(1 - D^{-\frac{1}{3(r-1)}}\right)\frac{N}{r},
\end{equation*}
and moreover the number of choices for the next edge in the $i$-th iteration until this point is~$\frac{ND}{r} p_i^r \left(1\pm D^{-1/4}\right)$ with~$p_i = 1 - \frac{ir}{N}$.

We construct a sparse paving matroid by running the random greedy matching process until~$m = \left(1 - (n-r+1)^{-\frac{1}{3(r-1)}}\right)\frac{N}{r}$ edges are selected. The (natural) logarithm of the number of choices is at least
\begin{equation*}
	\begin{split}
		\sum_{i=1}^m \ln\left(\frac{ND}{r} p_i^r\left(1 - D^{-\frac{1}{4}}\right)\right)
			& = m \left(\ln \frac{ND}{r} - 2D^{-\frac{1}{4}}\right) + r \sum_{i=1}^m \ln\left(1-\frac{ir}{N}\right) \\
			&\ge m\left(\ln \frac{ND}{r} - r - 2D^{-\frac{1}{4}}\right).
	\end{split}
\end{equation*}

Of course, the order in which blocks are chosen does not matter, so we must subtract
\begin{equation*}
	\ln m! = m\left(\ln m - 1 + O\left(\frac{\ln m}{m}\right)\right),
\end{equation*}

It follows that~$\ln s(n,r)$ is at least
\begin{equation*}
	m\left(\ln D - r + 1 - 2D^{-1/4} - O\left(\frac{\ln m}{m}\right)\right) = \frac{1}{n-r+1}\binom{n}{r}\left(1-D^{-\frac{1}{3(r-1)}}\right)\left(\ln (\e^{1-r} D) - o(1)\right),
\end{equation*}
from which the theorem follows.

\subsection{Entropy counting}

In this section, we will briefly review the definitions of entropy that we need. For a more thorough introduction to entropy, see e.g.\ \cite[Section~15.7]{AlonSpencer2008}.

Let~$X$ be a random variable taking values in a finite set~$S$. The {\em entropy} of~$X$ is defined as
\begin{equation*}
	\entropy{X} \defeq -\sum_{x \in S} \prob{}{X=x}\log\prob{}{X=x}.
\end{equation*}
It can be shown that~$\entropy{X} \le \log|S|$, with equality if and only if~$X$ has the uniform distribution on~$S$ (i.e.\ $\prob{}{X=x} = \frac{1}{|S|}$ for all~$x\in S$).

We are interested in the case where~$X = (X_i : i \in I)$ is a random vector, indexed by some index set~$I$. In this situation, for~$A \subseteq I$, we write~$X_A = (X_i : i \in A)$ for the {\em projection} of~$X$ onto the coordinates index by~$A$. Shearer's lemma~\cite{ChungGrahamFranklShearer1986} asserts that if~$\mathcal{A}$ is a collection of such index sets, and each~$i \in I$ is contained in at least~$k$ members of~$\mathcal{A}$, then
\begin{equation*}
	\entropy{X} \le \frac{1}{k} \sum_{A \in \mathcal{A}} \entropy{X_A}.
\end{equation*}

As a first application, we will use Shearer's lemma to put an upper bound on the number of sparse paving matroids. In the proof of this result, $X$ will be a random subset of the vertices of the Johnson graph~$J(n,r)$. Equivalently, we view~$X$ as a random binary vector, indexed by the vertices of~$J(n,r)$, in which~$X_v = 1$ if and only if~$v$ is selected into the random subset.
\thmsnrentropy*

\begin{proof}
	Let~$X$ be a random stable set of~$J(n,r)$, selected uniformly at random, so~$\log s(n,r) = \entropy{X}$. In order to apply Shearer's lemma, let~$\mathcal{A}$ be the collection of $(n-r+1)$-cliques consisting of vertices~$B\cup\{v\}$, $v\not\in B$, where~$B$ ranges over the $(r-1)$-subsets of~$[n]$. There are~$\binom{n}{r-1}$ such subsets, and each vertex is contained in exactly~$r$ of them. If~$A\in\mathcal{A}$, then~$X_A$ takes one of~$n-r+2$ values, as each clique~$A$ contains at most one vertex of~$X$. The theorem now follows from Shearer's lemma, as
	\begin{equation*}
		\entropy{X} \le \frac{1}{r} \sum_{A\in\mathcal{A}} \entropy{X_A} \le \frac{1}{r} \binom{n}{r-1} \log(n-r+2) = \frac{1}{n-r+1} \binom{n}{r} \log(n-r+2). \qedhere
	\end{equation*}
\end{proof}

A second application of Shearer's lemma can be used to leverage bounds on the number of paving matroids of small rank to prove bounds on the number of paving matroids of general rank.
\begin{lemma}\label{lemma:blowup}
	Let~$0 \le t \le r \le n$. Let~$\mathcal{M}$ be a class of matroids that is closed under contraction and isomorphism, and write~$\tilde{m}(n,r) \defeq \#\{M \in \mathcal{M} : E(M) = [n], r(M) = r\}$. Then~(i)
	\begin{equation*}
		\log (\tilde{m}(n,r) + 1) \le \frac{\binom{n}{r}}{\binom{n-t}{r-t}} \log(\tilde{m}(n-t, r-t) + 1),
	\end{equation*}
	and (ii) if no~$M \in \mathcal{M}$ of rank~$r$ has a circuit of size at most~$t$, then
	\begin{equation*}
		\log \tilde{m}(n,r) \le \frac{\binom{n}{r}}{\binom{n-t}{r-t}} \log \tilde{m}(n-t,r-t).
	\end{equation*}
\end{lemma}

The first statement in the lemma was proved in~\cite[Lemma~2]{BPvdP2014}. The second case follows from the same argument, after realising if~$X^{E,r}$ is the set of bases of a matroid~$M$ in which every $t$-set~$T$ is independent, then~$X^{E,r}/T$ is never empty.

\subsection{Matroids of rank 2}

Matroids on~$[n]$ of rank 2 are in 1-1 correspondence with partitions of~$[n+1]$ with at least three blocks (where e.g.\ the block containing~$n+1$ identifies the set of loops of the matroid, and the other blocks identity parallel classes). Partitions are counted by {\em Bell numbers}, so~$m(n,2) = B(n+1) - 2^n$.

Paving matroids on~$[n]$ of rank 2 are in 1-1 correspondence with 1-partitions of~$[n]$, which are just nontrivial ordinary partitions. It follows that~$p(n,2) = B(n) - 1$.

Asymptotics for Bell numbers are given e.g.\ in~\cite{MoserWyman1955}, where it is shown that $\log B(n) \sim n \log n$.

Sparse paving matroids on~$[n]$ of rank 2 are in 1-1 correspondence with stable sets in the Johnson graph~$J(n,2)$, as well as to matchings in the complete graph~$K_n$, and the number of involutions on~$[n]$. The number~$s(n,2)$ is therefore given by the {\em telephone number}~$T(n)$. The exponential generating function of the telephone numbers is known to be~$z \mapsto \exp(z+z^2/2)$, and precise asymptotics of its coefficients are known~\cite[p.\ 64]{TAOCP3}: We obtain $\log s_{n,2} \sim \frac{n}{2}{\log n}$.

The next theorem summarises the results in this section.
\thmrank*

\subsection{Binomial coefficients}

We will often make use, without stating so explicitly, of the following inequality on sums of binomial coefficients:
\begin{equation*}
	\sum_{i=0}^r \binom{n}{i} \le \left(\frac{\e n}{r}\right)^r, \qquad\text{for all $0 < r \le n$.}
\end{equation*}



\section{Truncation and erection of matroids}\label{sec:truncation_erection}

Let~$M$ be a matroid on ground set~$E$ with independent sets~$\mathcal{I}$. The {\em rank-$k$ truncation} of a matroid is the matroid~$M^{(k)}=\left(E, {\mathcal I}^{(k)}\right)$, where 
\begin{equation*}
	\mathcal{I}^{(k)}\defeq\left\{ I\in {\mathcal I}: |I|\leq k\right\}.
\end{equation*}
If~$M$ has rank~$r$, then its {\em truncation} is defined as $T(M)\defeq M^{(r-1)}$. Following Crapo~\cite{Crapo1970} we call~$N$ an {\em erection} of~$M$ if~$T(N)=M$ or~$N=M$ (in the latter case, we say that~$N$ is the trivial erection of~$M$). Note that while the truncation of a matroid is unique, a matroid may have many erections.

A set~$X\subseteq E$ is {\em $k$-closed} in~$M$ if~$\cl(Y)\subseteq X$ for each~$Y\subseteq X$ such that~$|Y|\le k$.  The {\em $k$-closure} of a set $X$ is
\begin{equation*}
	\cl_k(X)\defeq\bigcap \{ Y\subseteq E: Y\text{ $k$-closed}, X\subseteq Y\}.
\end{equation*}
Clearly the intersection of any collection of $k$-closed sets will be $k$-closed, so that  $\cl_k(X)$ is necessarily $k$-closed.

The following theorem, due to Crapo~\cite{Crapo1970}, identifies erections of matroids in terms of their rank-$r$ flats.
\begin{theorem}[{\cite[Theorem~2]{Crapo1970}}]\label{thm:crapo}
	Let~$M$ be a matroid on~$E$ of rank~$r$ and let~$\mathcal{H} \subseteq 2^E$. Then $\mathcal{H}$ is the set of rank-$r$ flats of an erection of~$M$ if and only if 
	\begin{enumerate}
		\item each $H\in \mathcal H$ has rank $r$ in $M$;
		\item each $H\in \mathcal H$ is $(r-1)$-closed; and
		\item each basis of $M$ is contained in a unique element of $\mathcal{H}$.
	\end{enumerate}
\end{theorem}

Note that if~$M$ is a matroid of rank-$r$, than the rank-$r$ flats of its erection are hyperplanes, unless the erection is trivial.

Independently, Knuth~\cite{Knuth1975} describes a procedure which, given a matroid~$M$ on~$E$ of rank~$r$ and a collection~$\mathcal{U}$ of subsets of~$E$, generates a set~$\mathcal{H}$ which are the rank-$r$ flats of an erection of~$M$ with the additional property that each~$U\in\mathcal{U}$ is contained in some rank-$r$ flat. We describe this procedure.

\hfill
\begin{minipage}{0.9\textwidth}
	\vspace{1em}\rule{\textwidth}{1.5pt}
	\begin{algorithmic}
		\AlgInput Matroid~$M$, collection of sets~$\mathcal{U}$
		\AlgOutput $\mathcal{H}$
		\Statex
		\State {\em Initialise:} $\mathcal{H} \leftarrow \mathcal{U} \cup \left\{F+e : \text{$F$ a hyperplane of $M$}, e \not\in F\right\}$
		\While{$\exists H, H' \in \mathcal{H}$ such that~$H \neq H'$, $r_M(H\cap H') = r$}
		\State {\em Update:} $\mathcal{H} \leftarrow \mathcal{H} \setminus\{H,H'\}\cup\{H \cup H'\}$
		\EndWhile
	\end{algorithmic}%
	\vspace{-1em}
	\rule{\textwidth}{1pt}
	\vspace{1em}
\end{minipage}\hfill{}


As~$|\mathcal{H}|$ decreases by~1 in each application of~{\em Update}, the procedure terminates after a finite number of steps on any input~$\mathcal{U} \subseteq 2^E$. It is not obvious, however, that the output of the procedure is independent of the choice of~$H$ and~$H'$ in each application of~{\em Update}.

\begin{theorem}[{\cite[Section~6]{Knuth1975}}]\label{thm:knuth}
	The output~$\mathcal{H}$ of the procedure depends only on the input~$M$ and~$\mathcal U$.
	Moreover, for any matroid~$M$ of rank~$r$ on~$E$, and any set~$\mathcal{H}\subseteq 2^E$, the following are equivalent:
	\begin{enumerate}
		\item $\mathcal{H}$ is the output of the procedure on some input~$\mathcal{U}\subseteq 2^E$; and   
		\item $\mathcal{H}$ is the set of rank-$r$ flats of an erection of~$M$. 
	\end{enumerate}
\end{theorem}

The results of Crapo and Knuth are closely related. To exhibit the relation between the two results, we will prove them both in a common framework.

With respect to a matroid~$M$ on~$E$ of rank~$r$, we call a collection~$\mathcal{X} \subseteq 2^E$ {\em complete} if
\begin{enumerate}[(i)]
	\item for all~$X\in \mathcal{X}$ and~$Y\subseteq X$, we have~$Y\in \mathcal{X}$;
	\item each basis of~$M$ is in~$\mathcal{X}$;
	\item for all~$X\in \mathcal{X}$, we have~$\cl_{r-1}(X)\in \mathcal{X}$; 
	\item for all~$X,Y\in \mathcal{X}$ such that~$r_M(X\cap Y)=r$, we have~$X\cup Y\in \mathcal{X}$.
\end{enumerate}

The following lemma describes how complete sets are in 1-1 correspondence with erections of~$M$.
\begin{lemma}\label{lem:comp}
	Let~$\mathcal{X}\subseteq 2^E$. Then~$\mathcal{X}$  is complete with respect to~$M$ if and only if 
	$\mathcal{X} = \{ X\subseteq E: r_N(X)\leq r\}$
	for some erection~$N$ of~$M$.
\end{lemma}
\begin{proof}
	Sufficiency is straightforward; we prove necessity.
	So assume that~$\mathcal{X}$  is complete with respect to~$M$, and let $r:2^E\rightarrow \Z$ be defined by
	\begin{equation*}
		r(X)\defeq
		\begin{cases}
			r_M(X)&\text{if }X\in \mathcal{X} \\
			r+1&\text{otherwise.}
		\end{cases}
	\end{equation*}
	If~$r$ is the rank function of a matroid, then it satisfies the conditions of the lemma.
	We verify that~$r$ satisfies the matroid rank axioms.
	Clearly, $r(X) \ge 0$ for all~$X\in E$. That~$r(Y) \le r(X)$ for all~$Y\subseteq X$ follows from the fact that~$\mathcal{X}$ is closed under taking subsets.
	It remains to show that~$r$ is submodular, i.e.
	\begin{equation}\label{eq:submodular}
		r(X)+r(Y)\geq r(X\cup Y)+r(X\cap Y).
	\end{equation}
	Suppose that~$r$ is not submodular. Pick~$X,Y\subseteq E$ violating~\eqref{eq:submodular} with $|X|+|Y|$ as large as possible.
	If~$X \not\in \mathcal{X}$, then~$X\cup Y\not\in \mathcal{X}$ so that~$r(X)+r(Y)= (r+1)+r(Y)\geq (r+1)+  r(X\cap Y)\geq  r(X\cup Y) + r(X\cap Y)$.
	So~$X\in \mathcal{X}$, and similarly~$Y\in \mathcal X$. As~$\mathcal{X}$ is closed under taking subsets, we have~$X\cap Y\in \mathcal{X}$ as well, and
	\begin{equation*}
		r(X\cup Y)>r(X)+r(Y)-r(X\cap Y)=r_M(X)+r_M(Y)-r_M(X\cap Y)\geq r_M(X\cup Y).
	\end{equation*}
	Hence~$X\cup Y\not \in \mathcal X$, and therefore~$r(X\cup Y)=r+1=r_M(X\cup Y)+1$. 
	If~$r_M(X\cap Y)=r$, then~$X\cup Y\in \mathcal{X}$ as~$\mathcal{X}$ is complete, a contradiction.
	It follows that~$r_M(X\cap Y)<r$. Hence
 	\begin{equation*}
		r_M(X)+r_M(Y)=r_M(X\cap Y)+ r_M(X\cup Y)\leq (r-1) + r,
	\end{equation*}
	so that~$r_M(X)<r$ or~$r_M(Y)<r$. Without loss of generality~$r_M(X)<r$.
	Then there is an independent set~$I$ such that~$\cl(I)=\cl_{r-1}(I)\supseteq X$. Let~$B$ be any basis of~$M$ containing~$I$. Then~$B\in \mathcal{X}$, and hence~$X' \defeq \cl_{r-1}(B)\in \mathcal X$ properly contains~$X$. Since~$X' \cup Y \not\in\mathcal{X}$, the pair~$X',Y$ contradicts maximality of~$|X|+|Y|$.
\end{proof}

Being closed under taking subsets, a complete set~$\mathcal{X}$ is fully determined by the set~$\mathcal{H}$ of its inclusionwise maximal elements through~$\mathcal{X} = \mathcal{H}^\downarrow$, where
\begin{equation*}
	\mathcal{H}^\downarrow\defeq\left\{X : \text{$X\subseteq H$ for some~$H \in \mathcal{H}$}\right\}.
\end{equation*}

\begin{lemma}\label{lemma:comp_crapo}
	Let~$\mathcal{X}\subseteq 2^E$ be closed under taking subsets and let~$\mathcal{H}$ be the set of inclusionwise maximal elements of~$\mathcal{X}$. Then~$\mathcal{X}$  is complete if and only if $\mathcal H$ satisfies conditions (1)--(3) of Crapo's Theorem.
\end{lemma}

A straightforward application of Lemma~\ref{lem:comp} and Lemma~\ref{lemma:comp_crapo} proves Crapo's theorem.

It is readily verified that if~$\mathcal{X}, \mathcal{Y}$ are both complete, then so is~$\mathcal{X}\cap \mathcal{Y}$. Hence the {\em completion}
\begin{equation}\label{eq:completion}
	\comp(\mathcal{Z}) \defeq \bigcap \left\{\mathcal{X} : \text{$X$ complete with respect to~$M$}, \mathcal{Z}\subseteq \mathcal{X}\right\}
\end{equation}
is a closure operator. (Note that the definition of~$\comp$ depends on~$M$.) The next lemma asserts that Knuth's procedure essentially determines $\comp(\mathcal{U})$.

\begin{lemma}
	Let~$\mathcal{U}\subseteq 2^E$, and let~$\mathcal{H}$ be the output of Knuth's procedure on input~$\mathcal{U}$.
	Then~$\mathcal{H}$ is the set of inclusionwise maximal elements of~$\comp(\mathcal U)$.
\end{lemma}
\begin{proof}
	Following the steps of Knuth's procedure, it is evident that if~$\mathcal{U}\subseteq \mathcal{X}$ and~$\mathcal{X}$ is complete, then $\mathcal{H}\subseteq \mathcal{X}$. Hence~$\mathcal{H}\subseteq \comp(\mathcal{U})$.
	On the other hand, a straightforward verification shows that if~$\mathcal{X}=\mathcal{H}^{\downarrow}$, then~$\mathcal{X}$ is complete, $\mathcal{H}$ is the set of maximal elements of~$\mathcal{X}$, and $\mathcal{U}\subseteq\mathcal{X}$. Hence~$\comp(\mathcal{U})\subseteq \mathcal{H}^{\downarrow}$.
\end{proof}

Knuth's theorem can now be argued as follows. First of all, the output~$\mathcal{H}$ on input~$\mathcal{U}$ is determined by~$M$ and~$\mathcal{U}$. It then follows from Lemma~\ref{lem:comp} that~$\mathcal{H}$ is the set of rank-$r$ flats in an erection of~$M$.

Arguing in the other direction, if~$\mathcal{H}$ is the collection of rank-$r$ flats in an erection of~$M$, then~$\mathcal{H}^\downarrow$ is complete by Lemma~\ref{lem:comp}, and as~$\comp(\mathcal{X}) = \mathcal{X}$ for complete sets~$\mathcal{X}$, it follows that~$\mathcal{H}$ is the output of Knuth's procedure on input~$\mathcal{H}^\downarrow$.

For the remainder of the paper, we fix some notation. We will write~$\mathcal{H}(M,\mathcal{U})$ for the outcome of Knuth's procedure on input~$\mathcal{U}$, and we will write~$M\uparrow\mathcal{U}$ for the erection of~$M$ of which the collection of rank-$r$ flats is given by~$\mathcal{H}(M,\mathcal{U})$. Thus~$r(M\uparrow\mathcal{U}) = r+1$, except when~$M\uparrow\mathcal{U} = M$, in which case~$\mathcal{H}(M,\mathcal{U}) = \{E\}$.

In similar terms Knuth describes the construction of
\begin{equation*}
	M(E, \mathcal{U}_0, \ldots, \mathcal{U}_k)\defeq(\cdots(M_0\uparrow \mathcal{U}_0)\uparrow\cdots )\uparrow \mathcal{U}_k,
\end{equation*}
where~$M_0$ denotes the rank-0 matroid on~$E$. Since for each matroid~$M$, there is some~$\mathcal{U}$ such that $M=M^{(r-1)}\uparrow \mathcal{U}$, 
the matroid~$M(E, \mathcal U_0, \ldots, \mathcal U_k)$ ranges over all possible matroids as~$(\mathcal{U}_0, \ldots, \mathcal{U}_k)$
ranges over all sequences of sets. 
Knuth proposes that to create a `random' matroid~$M$, one should generate~$M=M(E, \mathcal U_0, \ldots, \mathcal U_k)$ from random sets~$\mathcal{U}_i$.

Both Crapo and Knuth comment on the potential for describing a matroid~$N$ concisely as~$M=T(M)\uparrow \mathcal{U}$, picking~$\mathcal{U}$ as small or as `simple' as possible.

Crapo notes that certain flats~$F$ of a matroid~$M$ are `predictable' in that~$M|F$, the restriction of~$M$ to~$F$, has no nontrivial erection (i.e.\ an erection of rank strictly greater than~$r_M(F)$). The existence of~$F$ as a flat of~$M$ of that rank is therefore unavoidable given the structure of~$M|F$. To describe the matroid, it would therefore suffice to point out the inpredictable or {\em essential} flats. Crapo notes the following condition for being essential.
\begin{theorem}[{\cite[Theorem 12]{Crapo1970}}]\label{thm:essential}
	Let~$M$ be a matroid on~$E$, and let~$F\subseteq E$ be a flat of~$M$ of rank~$k$. If~$\cl_{k-1}(I)=F$ for some independent set~$I$ of~$M$, then~$F$ is not essential.  
\end{theorem}
Crapo illustrates that the list of essential flats can give a very concise description of a matroid, and notes that this phenomenon was first observed by Higgs.

\subsection{Historical remarks}

Matroid erections were first studied by Crapo~\cite{Crapo1970}, who recognised the lattice structure of erections under the {\em weak order} (if~$M$ and~$N$ are matroids on the same groundset, then~$M$ is said to be less than~$N$ in the weak order if every independent set in~$N$ is also independent in~$M$).

Duke~\cite{Duke1987} showed that the lattice of erections is isomorphic to the interval~$[\mathcal{M}_0, \mathcal{M}_c]$ in the lattice of modular cuts of~$M\dual$, ordered by the superset-relation. Here, $\mathcal{M}_0$ is the set of all flats of~$M\dual$, and~$\mathcal{M}_c$ is the modular cut generated by the cyclic flats of~$M\dual$.
If~$M$ and~$L$ are matroids on the same groundset, then~$L$ is called a lift of~$M$ if there is a matroid~$M'$ such that~$M = M'/e$ and~$L = M'\backslash e$. As~$(M' / e)\dual = (M')\dual \backslash e$, the lifts of~$M$ are in 1-1 correspondence with single-element extensions of~$M\dual$. Duke's result asserts that those lifts that are erections correspond precisely with modular cuts in~$[\mathcal{M}_0, \mathcal{M}_c]$.

We can link complete sets directly to modular cuts~$\mathcal{M}$ in the interval~$[\mathcal{M}_0, \mathcal{M}_c]$ through
\begin{equation*}
	\mathcal{X}\left(\mathcal{M}\right) \defeq \left\{A \subseteq E(M) : \cl^*(E\setminus A) \in \mathcal{M}\right\}.
\end{equation*}
It is readily verified that~$\mathcal{X}(\mathcal{M})$ is a complete set for all~$\mathcal{M} \in [\mathcal{M}_0, \mathcal{M}_c]$. In the other direction, $\mathcal{M}(\mathcal{X})\defeq\{\cl^*(E\setminus A) : A\in\mathcal{X}\}$ provides the inverse of~$\mathcal{X}(\mathcal{M})$.

The collection of modular cuts of~$M\dual$ (and hence of the interval~$[\mathcal{M}_0, \mathcal{M}_c]$) is closed under taking arbitrary intersections, so every set of flats of~$M\dual$ generates a smallest modular cut containing the set. The completion operator~\eqref{eq:completion} is a specialisation of this closure operator for modular cuts.

\section{Matroid erection and paving matroids}\label{sec:erection_paving}

\subsection{Reconstructing a matroid from its truncation}

The truncation of a matroid of rank~$r$ is obtained by declaring all sets of cardinality at least~$r$ to be dependent. The following lemma illustrates how this operation can be reversed. In particular, it shows how the original matroid can be reconstructed from its truncation and some extra information.
\begin{lemma}\label{Ur} 
	Let $M$ be a matroid of rank $r$ of $E$, and let $\mathcal H$ be the set of hyperplanes of $M$. Suppose that for each hyperplane $H\in\mathcal H$, $U_H$ is  such that $\cl_{r-2}(U_H)=H$. Let
\begin{equation*}
	\mathcal U\defeq\left\{ U_H: H \text{ a hyperplane of }M,  |U_H|\geq r \right\}.
\end{equation*}
Then $M=M^{(r-1)}\uparrow \mathcal U$.
\end{lemma}
\begin{proof}
	For each hyperplane $H\in\mathcal H$ of $M$, either $U_H\in \mathcal U$ or $U_H$ is a basis of $M^{(r-1)}$. Let
	\begin{equation*}
		\mathcal W = \left\{ \cl_{r-2}(X) : X \text{ is a basis of }M^{(r-1)}, \text{ or }X\in \mathcal U \right\}.
	\end{equation*}
	Then $\mathcal W$ includes all hyperplanes, and in addition some sets $\cl_{r-2}(X)$ contained in a hyperplane $\cl(X)$ of $M$. So the hyperplanes of $M$ are the inclusion-wise maximal elements of $\mathcal W$. In Knuth's procedure for computing $\mathcal H(M^{(r-1)}, \mathcal U)$, the initial step generates $\mathcal W$, and the remaining steps remove any non-maximal elements of $\mathcal W$. Hence $\mathcal H=\mathcal H(M^{(r-1)}, \mathcal U)$.
\end{proof}

We say that a set of elements $U\subseteq E$ is {\em $k$-free} if there is no circuit $C$ with $|C|\leq k$ such that $C\subseteq U$.
\begin{lemma}
	Let $X\subseteq E$ be $(k-1)$-closed. There exists a $k$-free set $U\subseteq X$ such that $X= \cl_{k-1}(U)$.
\end{lemma}
\begin{proof}
	Pick any inclusionwise minimal set $U\subseteq X$ such that $\cl_{k-1}(U)= X$.
\end{proof}

\begin{lemma}\label{pav}
	Let $M$ be a matroid of rank $r$ of $E$, and let $\mathcal H$ be the set of hyperplanes of $M$. Suppose that for each hyperplane $H\in\mathcal H$, $U_H$ is a $k$-free set such that $\cl_{r-2}(U_H)=H$, and put 
	\begin{equation*}
		\mathcal{U} \defeq \{ U_H: H \text{ a hyperplane of }M,  |U_F|\geq r \}.
	\end{equation*}
	Then $\mathcal{U}$ is the set of dependent hyperplanes of a paving matroid of rank $r$ on $E$.
\end{lemma} 
\begin{proof}
	We must show that each $(r-1)$-subset of $E$ is contained in at most one element of $\mathcal U$. If not, there are distinct hyperplanes $H, H'$ of $M$ so that $|U_H\cap U_{H'}|\geq r-1$. But $r(U_H\cap U_{H'})\leq r(H\cap H')\leq r-2$, and as a subset of the $(r-1)$-free set $U_H$, $U_H\cap U_{H'}$ is $(r-1)$-free. Hence $U_H\cap U_{H'}$ is independent, and thus $|U_H\cap U_{H'}|=r(U_H\cap U_{H'})\leq r-2$.
\end{proof}
It follows from this lemma that the number of erections of a matroid $M$ on $n$ elements and of rank $r$ does not exceed $p(n,r+1)$, the number of erections of the uniform matroid $U(r,n)$. Moreover, it follows that each such matroid can be written as~$M=T(M)\uparrow\mathcal{U}$, where~$\mathcal{U}$ is the set of dependent hyperplanes of a paving matroid of rank~$r$ on~$E$. This immediately gives the following result.
\begin{lemma}\label{lemma:mnr_pnr}
	$m(n,r) \le m(n,r-1)p(n,r)$.
\end{lemma}

\thmfixrank*
\begin{proof}
	It follows from Lemma~\ref{lemma:mnr_pnr} that
	\begin{equation*}
		\log m(n,r) \le \log p(n,r) + \log m(n,r-1) \le \log p(n,r) + \frac{\log n}{n} \binom{n}{r-1}(1+o(1)).
	\end{equation*}
	The result now follows upon comparing the right-most term with the lower bound
	\begin{equation*}
		\log p(n,r) \ge \log s(n,r) \ge \frac{\log n}{n} \binom{n}{r}(1-o(1))
	\end{equation*}
	from Theorem~\ref{thm:general_keevash}.
\end{proof}

\subsection{A bound on the number of paving matroids}

Recall that every paving matroid is the erection of a uniform matroid. In this section, we provide a bound on the amount of information that is necessary to specify any paving matroid. This translates to an upper bound on the number of paving matroids.

\thmpupperbound*

\begin{proof}
	If~$M$ is a paving matroid, then the truncation of~$M$ is a uniform matroid. We will show that for each paving matroid~$M$ on~$E=[n]$ of rank~$r$, there is a set of~$\mathcal{V} \subseteq \binom{E}{r}$ such that~$M=U(r-1,n)\uparrow \mathcal{V}$ with~$|\mathcal{V}| \le \frac{1}{n-r+1} \binom{n}{r}$. It then follows that~$p(n,r)$ is bounded by the number of collections~$\mathcal{V} \subseteq \binom{E}{r}$ of at most the given cardinality, i.e.\
\begin{equation*}
		p(n,r) \le \sum_{i=0}^{\frac{1}{n-r+1} \binom{n}{r}} \binom{\binom{n}{r}}{i} \le \left(\e(n-r+1)\right)^{\frac{1}{n-r+1}\binom{n}{r}},
	\end{equation*}
	and the theorem follows upon taking logarithms.

	We construct~$\mathcal{V}$. For each hyperplane~$H$ of~$M$, put
	\begin{equation*}
		\mathcal{V}(H) = \left\{V \in \binom{H}{r} : \text{$V$ is consecutive in $H$}\right\},
	\end{equation*}
	where {\em consecutive in~$H$} means that there are no~$e, f \in V$ and $g \in H \setminus V$ such that~$e < g < f$ (recall that~$E=[n]$, so in particular it is linearly ordered). Let~$\mathcal{V}\defeq\cup_{H \in \mathcal{H}(M)} \mathcal{V}(H)$. Then~$\comp(\mathcal{V}) = \mathcal{H}(M)^{\downarrow}$, so that~$M = U(r-1,n)\uparrow \mathcal{V}$.

	To bound the cardinality of~$\mathcal{V}$, note that if~$H$ is a hyperplane of cardinality~$k$, then~$|\mathcal{V}(H)| = k-r+1$, so if we write~$h_k$ for the number of hyperplanes of cardinality~$k$, we obtain~$|\mathcal{V}| \le \sum_{k=r-1}^n (k-r+1)h_k$. As~$M$ is paving, each~$r-1$-set is contained in a unique hyperplane. It follows that
	\begin{equation*}
		\sum_{k = r-1}^n h_k \binom{k}{r-1} = \binom{n}{r-1}.
	\end{equation*}
	Maximising over the~$h_k$, we obtain
	\begin{equation*}
		|\mathcal{V}| \le \max\left\{ \sum_{k=r}^n (k-r+1) h_k : \sum_{k=r}^n h_k \binom{k}{r-1} = \binom{n}{r-1}, h_k \ge 0\right\} \le \frac{1}{r}\binom{n}{r-1} = \frac{1}{n-r+1}\binom{n}{r},
	\end{equation*}
	since~$\min\limits_{k \ge r} \frac{1}{k-r+1}\binom{k}{r-1} = r$.
\end{proof}

Upon combining the lower bound on~$s(n,r)$ from Theorem~\ref{thm:general_keevash} with the upper bounds on~$p(n,r)$ of Theorem~\ref{thm:p_upperbound} and~$m(n,r)$ of Theorem~\ref{thm:fixrank}, we find that for all fixed~$r \ge 3$
\begin{equation*}
	\log m(n,r) \sim \log p(n,r) \sim \log s(n,r) \qquad\text{as $n\to\infty$.}
\end{equation*}

\section{Two antichains associated with matroids}\label{sec:antichains}

\subsection{Antichains and the LYM-inequality}

Let~$E$ be a set. A collection of~$\mathcal{A} \subseteq 2^E$ is called an {\em antichain} (or {\em clutter}) if no element of~$\mathcal{A}$ is properly contained in another, i.e.\ if for all~$A,B\in\mathcal{A}$, $A \subseteq B$ implies~$A = B$.

A central result in the theory of antichains is the well-known {\em LYM-inequality}. For a proof, we refer to~\cite[Theorem~9.6]{Jukna2001}.

\begin{lemma}[LYM-inequality]
	Let~$E$ be a set of cardinality~$n$, and let~$\mathcal{A}$ be an antichain. Then
	\begin{equation*}
		\sum_{A \in \mathcal{A}} \frac{1}{\binom{n}{|A|}} \le 1.
	\end{equation*}
\end{lemma}

To each matroid, a number of antichains are naturally associated: the collections of circuits, bases, and hyperplanes are each antichains. In this section, we construct two further antichains that are associated with matroids, both of which give a succinct description of matroids.

\subsection{The antichain~$\mathcal{A}(M)$}

Let~$M$ be a matroid of rank~$r$ on~$E$. In Lemma~\ref{Ur} we constructed a collection~$\mathcal{U}$ so that~$M = M^{(r-1)} \uparrow \mathcal{U}$. There is nothing that stops us from giving the truncation~$M^{(r-1)}$ the same treatment, and writing~$M^{(r-1)} = M^{(r-2)} \uparrow \mathcal{U}'$. Continuing in this fashion, we inductively obtain a sequence of erections of the rank-0 matroid on~$E$.

Now, if~$F$ is a flat of rank~$k$ in~$M$, let~$U_F$ be such that~$\cl_{k-1}(U_F) = F$, and for~$k = 0, 1, \ldots, r-1$ put
\begin{equation}\label{eq:Uk}
	\mathcal{U}_k \defeq \{U_F : \text{$F$ a flat of rank $k$}, |U_F| > k\}.
\end{equation}
Then~$M = M(E,\mathcal{U}_0, \mathcal{U}_1, \ldots, \mathcal{U}_{r-1})$.

The sets~$\mathcal{U}_k$ in~\eqref{eq:Uk} do not include a representative for each flat~$F$: if~$F$ is independent, then~$U_F$ is omitted. If it is our objective to determine~$M = M(E, \mathcal{U}_0, \mathcal{U}_1, \ldots, \mathcal{U}_{r-1})$ as economically as possible, then a sensible first optimisation is to pick each set~$U_F$ as small as possible under the condition that~$\cl_{k-1}(U_F) = F$.

Let us assume that the groundset~$E$ is linearly ordered. This induces the {\em graded lexicographic order} on~$2^E$, for which we say that~$X$ precedes~$Y$ (denoted by~$X\prec Y$) if
\begin{equation*}
	|X| < |Y|, \quad\text{or}\quad |X| = |Y| \text{ and } \min (X \triangle Y) \in X.
\end{equation*}
In what follows, if~$\mathcal{X} \subseteq 2^E$ is not the empty set, then~$\min \mathcal{X}$ refers to the smallest member of~$\mathcal{X}$ in the graded lexicographic order.

For a flat of rank~$k$, we can now define
\begin{equation*}
	U_F^* \defeq \min\{U : \cl_{k-1}(U) = F\}.
\end{equation*}
Then~$U_F^*$ are inclusionwise minimal, and in particular it is $k$-free.

We define
\begin{equation*}
	\mathcal{U}_k^* \defeq \{U_F^* : \text{$F$ is a flat of rank~$k$}, |U_F^*| > k\},
	\qquad\text{and}\qquad
	\mathcal{U}^*(M) \defeq \bigcup_{k=0}^{r-1} \mathcal{U}_k^*.
\end{equation*}

\begin{lemma}
	$\mathcal{U}^*(M)$ is an antichain.
\end{lemma}
\begin{proof}
	Suppose that there exist~$U, U' \in \mathcal{U}^*(M)$ so that~$U\subseteq U'$. If~$r(U)=k$, then~$|U| > k$, so~$U$ is dependent. If~$r(U')=k'$, then~$U'$ is $k'$-free, and as~$U'$ contains the dependent set~$U$, it follows that~$k \ge k'$. On the other hand, $k = r(U) \le r(U') = k'$, so in fact~$k = k'$. But then~$F = \cl_{k-1}(U) \subseteq \cl_{k-1}(U') = F'$, for some flats~$F, F'$. These flats both have rank~$k$, so~$F=F'$, hence~$U = U_F^* = U'$.
\end{proof}

\begin{lemma}\label{lemma:Uk_subset}
	Let $F, F'$  be flats of $M$ of ranks $k,k'$ respectively, such that~$k'>k$, and let $X\subseteq U^*_{F}$ have cardinality $k$. 
	If $X\subseteq U^*_{F'}$, then $X=\min \binom{U^*_{F}}{k}$.
\end{lemma}
\begin{proof}
	Suppose that this is not the case. Let~$X^* = \min \binom{U_F^*}{k}$, and let~$U' = (U_{F'}^* \setminus X) \cup X^*$. Since~$k'>k$ we have~$\cl_{k'-1}(X) = \cl(X) = \cl_{k'-1}(X^*)$, and so $\cl_{k'-1}(U') = F'$. Moreover, we have~$U' \prec U_F^*$, thus contradicting minimality of~$U_{F'}^*$.
\end{proof}

For~$U \in \mathcal{U}_k^*$ we define
\begin{equation*}
	\mathcal{A}(U) \defeq
	\begin{cases}
		\binom{U}{r-1} & \text{if $k=r-1$} \\
		\binom{U}{k}^- & \text{if $0 < k < r-1$} \\
		\binom{U}{1} & \text{if $k = 0$,}
	\end{cases}
	\qquad\text{where}\qquad
	\binom{U}{k}^- = \binom{U}{k} \setminus \left\{\min \binom{U}{k}\right\},
\end{equation*}
and we construct
\begin{equation*}
	\mathcal{A}(M) \defeq \bigcup_{k=0}^{r-1} \bigcup_{U \in \mathcal{U}_k^*} \mathcal{A}(U).
\end{equation*}

\begin{lemma}
	$\mathcal{A}(M)$ is an antichain, and~$\mathcal{A}(U) \cap \mathcal{A}(U') = \emptyset$ for distinct~$U, U' \in \mathcal{U}^*(M)$.
\end{lemma}

\begin{proof}
	The first statement in the lemma is an immediate consequence of Lemma~\ref{lemma:Uk_subset}, and the fact that~$U \in \mathcal{U}_k^*$ does not contain any loops when~$k>0$.

	For the second statement, let~$U \in \mathcal{U}_i^*$ and $U' \in \mathcal{U}_j^*$ correspond to the flats~$F$ and~$F'$; without loss of generality, we may assume~$i\le j$. Suppose that there exists~$X \in \mathcal{A}(U)\cap\mathcal{A}(U')$. Suppose that~$|X|=k$. Then either~$j=k=1$ and~$i=0$, or~$i=j=k$. The former case cannot happen, as in that situation the single element in~$X$ must be a loop as well as a nonloop. In the latter case, as~$U$ is $k$-free, $X$ is independent, and hence~$F' = \cl(X) = F$. It follows that~$U=U'$.
\end{proof}

As~$\mathcal{A}(M)$ is an antichain, we can apply the LYM-inequality to it in order to obtain
\begin{equation}\label{eq:lym-A}
	\sum_{k=1}^{r-1} \sum_{U \in \mathcal{U}_k^*} \frac{|\mathcal{A}(U)|}{\binom{n}{k}} + \sum_{U \in \mathcal{U}_0^*} \frac{|\mathcal{A}(U)|}{\binom{n}{1}} \le 1.
\end{equation}

\thmessentialflatsbound*

\begin{proof}
	By Theorem~\ref{thm:essential}, we have~$|U_F^*| > r(F)$ for each essential flat of~$M$, so the cardinality of~$\mathcal{U}^*(M)$ is an upper bound for the number of essential flats.

	First, consider the case~$n=6,r=3$. Suppose that a matroid on~6 elements of rank~3 has~$p$ trivial parallel classes of points, and~$p'$ nontrivial parallel classes of points. We have~$|\mathcal{U}_0^*| \le 1$, $|\mathcal{U}_1^*| \le p'$, and~$|\mathcal{U}_2^*| \le p + p' - 2$. As $p + 2p' \le 6$, it follows that~$|\mathcal{U}^*(M)| \le 5$, which proves the theorem in this case.

	Thus, for the remainder of the proof, we may assume that~$n\ge\max\{2r,7\}$.
	Starting from~\eqref{eq:lym-A}, we obtain the related inequality
	\begin{equation*}
		\sum_{k=0}^{r-1} |\mathcal{U}_{k}^*| c_k \le 1
		\qquad\text{with}\qquad
		c_k =
		\begin{cases}
			\frac{r}{\binom{n}{r-1}} & \text{if $k = r-1$} \\
			\frac{|\cl(\emptyset)|}{n} & \text{if $k=0$} \\
			\frac{k}{\binom{n}{k}} & \text{if $0 < k < r-1$,}
		\end{cases}
	\end{equation*}
	since for each~$0 < k < r-1$ every~$U \in \mathcal{U}_k^*$ contributes at least~$k$ sets to~$\mathcal{A}(U)$, while every~$U \in \mathcal{U}_{r-1}^*$ contributes at least~$r$ sets to~$\mathcal{A}(U)$; the coefficient~$c_0$ follows from the fact that~$\mathcal{U}_0^* = \{\cl(\emptyset)\}$, and~$|\mathcal{A}(\cl(\emptyset))| = |\cl(\emptyset)|$.

	For~$0<k<r-2$ we have
	\begin{equation*}
		\frac{c_{k+1}}{c_k} = \frac{(k+1)^2}{k(n-k)} \le 1,
		\qquad\text{while}\qquad
		\frac{c_{r-1}}{c_{r-2}} = \frac{r(r-1)}{(r-2)(n-r+2)} \le 1,
	\end{equation*}
	so~$c_{r-1} \le c_k$ for all~$k > 0$. (It is at this point that we require that~$n\ge\max\{2r,7\}$.)

	If~$M$ does not have any loops, we have~$|\mathcal{U}_0^*| = 0$, and it follows that
	\begin{equation*}
		|\mathcal{U}^*(M)| = \sum_{k=1}^{r-1} |\mathcal{U}_k^*| \le \frac{1}{r}\binom{n}{r-1} = \frac{1}{n-r+1}\binom{n}{r}.
	\end{equation*}
	Similarly, if~$M$ does contain loops, we have~$c_0 \ge c_1 \ge c_{r-1}$ as well, and hence
	\begin{equation*}
		|\mathcal{U}^*(M)| = \sum_{k=0}^{r-1} |\mathcal{U}_k^*| \le \frac{1}{r}\binom{n}{r-1} = \frac{1}{n-r+1}\binom{n}{r}. \qedhere
	\end{equation*}
\end{proof}

Similar bounds as in Theorem~\ref{thm:essential-flats-bound} hold for matroids of rank~1 and rank~2. If~$r=1$, there is at most one essential flat, $\cl(\emptyset)$, so the result holds. If~$r=2$, and~$M$ has no loops, then the number of essential flats equals the number of nontrivial parallel classes, which is at most~$\frac{n}{2}$. If~$M$ has~$\ell>1$ loops, then~$|\mathcal{U}_0^*| = 1$, while~$|\mathcal{U}_1^*| \le \frac{n-\ell}{2}$. It follows that a matroid~$M$ of rank~2 has at most~$\frac{n}{2}$ essential flats, unless it is isomorphic to the matroid with exactly one loop and~$\frac{n}{2}$ paralles classes, each containing two points, in which case it has~$\frac{n+1}{2}$ essential flats.

\begin{corollary}\label{crl:mnr_bound-A}
	For all~$r \ge 3$ and~$n \ge 2r$, $\log m(n,r) \le \frac{1}{n-r+1} \binom{n}{r} \log \frac{\e 2^n (r+1) (n-r+1)}{\binom{n}{r}}$.
\end{corollary}

\begin{proof}
	Each matroid of rank~$r$ on~$[n]$ is determined by the set of pairs~$(F,r(F))$, with~$F$ ranging over the essential flats. By Theorem~\ref{thm:essential-flats-bound}, these pairs form a subset of~$2^{[n]} \times \{0, 1, \ldots, r\}$ of cardinality at most~$\frac{1}{n-r+1}\binom{n}{r}$. It follows that
	\begin{equation*}
		m(n,r) \le \sum_{i=0}^{\frac{1}{n-r+1}\binom{n}{r}} \binom{2^n (r+1)}{i} \le \left(\frac{\e 2^n (r+1) (n-r+1)}{\binom{n}{r}}\right)^{\frac{1}{n-r+1} \binom{n}{r}},
	\end{equation*}
	which gives the desired result after taking logarithms.
\end{proof}

Since~$m(n) = m(n,0) + m(n,1) + \ldots + m(n,n)$, Corollary~\ref{crl:mnr_bound-A} translates to a bound on all matroids. We thus obtain
\begin{equation*}
	\log m(n) \le \max_{r \le \lfloor n/2\rfloor} \log m(n,r) + \log (n+1) \le \max_{r\le \lfloor n/2\rfloor} \frac{1}{n-r+1} \binom{n}{r} \log \frac{\e 2^n (r+1) (n-r+1)}{\binom{n}{r}} + \log (n+1),
\end{equation*}
where we use that~$m(n,n-r) = m(n,r)$ by duality. It can be shown that the maximum is attained at~$r = \lfloor n/2\rfloor$, provided~$n$ is sufficiently large, and using that~$\binom{n}{\lfloor n/2\rfloor} = \Theta\left(\frac{2^n}{\sqrt{n}}\right)$ we obtain
\begin{equation*}
	\log m(n) \le \frac{3+o(1)}{n} \binom{n}{\lfloor n/2\rfloor} \log n.
\end{equation*}
Combining the best known lower bound on the number of matroids (obtained by Graham and Sloane~\cite[Theorem~1]{GrahamSloane1980} (as pointed out by Mayhew and Welsh~\cite{MayhewWelsh2013})) with the best known upper bound (obtained by Bansal and the current authors~\cite{BPvdP2015}), we obtain
\begin{equation*}
	\frac{1}{n}\binom{n}{\lfloor n/2\rfloor} \le \log m(n) \le \frac{2+o(1)}{n}\binom{n}{\lfloor n/2\rfloor},
\end{equation*}
so the bound obtained by essential flats is only a factor~$\Theta(\log n)$ off the actual value. Thus, Corollary~\ref{crl:mnr_bound-A} supports Higgs's suggestion that the essential flats of a matroid give a concise description of the matroid.

\subsection{The antichain~$\mathcal{V}(M)$}\label{ss:V-encoding}

The bound on the number of matroids in Corollary~\ref{crl:mnr_bound-A} is wasteful when we consider matroids of fixed rank~$r$. If this is the case, we have
\begin{equation*}
	\frac{1}{n-r+1} \binom{n}{r} \log \frac{\e 2^n (r+1) (n-r+1)}{\binom{n}{r}} \ge \binom{n}{r}(1-o(1)),
\end{equation*}
which gives a much weaker bound on~$\log m_{n,r}$ than suggested by e.g.~\eqref{eq:fixed_rank_equivalence}. The reason seems to be that~$U \in \mathcal{U}^*(M)$ can be arbitrarily large, which gives rise to the factor~$2^n$ inside the logarithm.

In this section, we resolve this problem by relating to~$\mathcal{U}^*(M)$ a different antichain~$\mathcal{V}(M)$, of which all members have bounded cardinality. In order to construct~$\mathcal{V}(M)$, for~$U \in \mathcal{U}_k^*$ define
\begin{equation*}
	\mathcal{V}(U) \defeq \left\{ V \in \binom{U}{k+1} : \text{$V$ is a consecutive in~$U$}\right\},
\end{equation*}
where a subset~$V \subseteq U$ is {\em consecutive} if there are no~$e,g \in V$ and~$f \in U\setminus V$ with~$e < f < g$, and we construct
\begin{equation*}
	\mathcal{V}_k \defeq \bigcup_{U \in \mathcal{U}_k^*} \mathcal{V}(U).
\end{equation*}

\begin{lemma}\label{lemma:V-encoding_circuit}
	Each~$V \in \mathcal{V}_k$ is a circuit of~$M$.
\end{lemma}

\begin{proof}
	Let~$U \in \mathcal{U}_k^*$ and let~$V \in \binom{U}{k+1}$ be a consecutive set. As~$U$ is~$k$-free, so is~$V$. However, $V$ has rank at most~$k$, so it must be dependent. It follows that~$V$ is a circuit.
\end{proof}

\begin{lemma}\label{lemma:reconstruction-V}
	$M = M(E, \mathcal{V}_0, \mathcal{V}_1, \ldots, \mathcal{V}_{r-1})$.
\end{lemma}

\begin{proof}
	If~$U\in\mathcal{U}_k$ gives~$\mathcal{V}(U) = \{V_1, V_2, \ldots, V_\ell\}$ with~$V_i \prec V_{i+1}$ for all~$i=1,2,\ldots,\ell-1$. If~$H = V_1 \cup \ldots \cup V_i$ and~$H' = V_{i+1}$, then~$r_{M^{(k)}}(H \cap H') = k$, as each element of~$\mathcal{V}(U)$ is a circuit of~$M^{(k)}$. Hence, $\mathcal{U}_k$ can be obtained from~$\mathcal{V}_k$ by a sequence of {\em Update}-steps of Knuth's algorithm. As~$M^{(k)}\uparrow\mathcal{V}_k$ does not depend on the order in which these steps are performed, it follows that~$M^{(k)}\uparrow\mathcal{V}_k = M^{(k)}\uparrow\mathcal{U}_k$ for all~$k$, and the result follows.
\end{proof}

Consequently, if~$\mathcal{V}_0 = \mathcal{V}_1 = \ldots = \mathcal{V}_{k-1} = \emptyset$, then the truncation~$M^{(k)}$ is the uniform matroid of rank~$k$ on~$E$.

Next, we define $\mathcal V(M) \defeq \bigcup_{k=0}^{r-1} \mathcal V_k$. Since the elements of each $\mathcal V_k$ are equicardinal, we can distinguish $\mathcal V_k=\{V\in \mathcal V(M): |V|=k+1\}$, and so $M$ is determined by $(E, r, \mathcal V(M))$ through the previous lemma.

Next, we record a number of useful properties of~$\mathcal{V}(M)$. The first of these is a consequence of Lemma~\ref{lemma:V-encoding_circuit}.

\begin{lemma}\label{lemma:V-encoding_circuit}
	$\mathcal V(M)$ is an antichain.
\end{lemma}

\begin{lemma}\label{lemma:V-properties}
	Let~$V, V' \in \mathcal{V}(M)$ with~$|V|=k+1$, $|V'| = k' +1$, and~$|V\cap V'| = k$.
	\begin{enumerate}[(i)]
		\begin{item}
			If~$k' > k$, then~$V\cap V' = \min\binom{V}{k}$. Moreover, if~$V \in \mathcal{V}(U_F^*)$, then~$V = \min\binom{U_F^*}{k+1}$.
		\end{item}
		\begin{item}
			If~$k' = k$ and~$V\prec V'$, then~$V \cap V' = \min\binom{V'}{k} = \max\binom{V}{k}$.
		\end{item}
	\end{enumerate}
\end{lemma}

\begin{proof}
	Suppose that~$F$, $F'$ are flats of rank~$k$, $k'$, respectively, so that~$V \in \mathcal{V}(U_F^*)$ and~$V' \in \mathcal{V}(U_{F'}^*)$. Note that~$V \cap V'$ is a $k$-subset of both~$U_F^*$ and~$U_{F'}^*$. If~$k' > k$, it follows from Lemma~\ref{lemma:Uk_subset} that~$V \cap V' = \min\binom{U_F^*}{k}$. There is only one consecutive $(k+1)$-set in~$U_F^*$ containing~$\min\binom{U_F^*}{k}$, and this is~$\min\binom{U_F^*}{k+1}$. This proves~(i). We obtain~(ii) as an immediate consequence of the definition of the graded lexicographic order.
\end{proof}

%
%
%

\begin{lemma}\label{lemma:lym-V}
	Suppose that~$r\ge3$ and~$n\ge 2r$. Then
		\begin{equation*}
		\sum_{k=0}^{r-1} |\mathcal{V}_k| c_k \le 1
		\qquad\text{with}\qquad
		c_k =
		\begin{cases}
			\frac{r}{\binom{n}{r-1}} & \text{if $k = r-1$} \\
			\frac{1}{n} & \text{if $k = 0$} \\
			\frac{k}{\binom{n}{k}} & \text{if $0 < k < r-1$.}
		\end{cases}
	\end{equation*}
\end{lemma}

\begin{proof}
	Our starting point is~\eqref{eq:lym-A}. The coefficients~$c_k$ follow from the following considerations. For all~$U \in \mathcal{U}_{r-1}^*$, we have~$|\mathcal{A}(U)| = \binom{|U|}{r-1}$ and~$|\mathcal{V}(U)| = |U|-r+1$. Consequently, $|\mathcal{A}(U)| \ge |\mathcal{V}(U)|r$. Similarly, if~$U \in \mathcal{U}_k^*$ for some~$0 < k < r-1$, we have~$|\mathcal{A}(U)| = \binom{|U|}{k}-1$, and~$|\mathcal{V}(U)| = |U| - k$, and we obtain~$|\mathcal{A}(U)| \ge |\mathcal{V}(U)|k$. Finally, for all~$U \in \mathcal{U}_0^*$ we have~$\mathcal{V}(U) = \mathcal{A}(U)$.
\end{proof}

A similar analysis as in the proof of Theorem~\ref{thm:essential-flats-bound} gives the following corollary.

\begin{corollary}\label{corollary:lym-V}
	Suppose that~$r\ge 3$ and~$n\ge2r$. Let~$M$ be a rank-$r$ matroid on a groundset of~$n$ elements, then~$|\mathcal{V}(M)| \le \frac{1}{n-r+1}\binom{n}{r}$.
\end{corollary}

In a Steiner system~$S(r-1,r,n)$, the number of blocks is~$\frac{1}{r}\binom{n}{r-1}$, from which it follows that Corollary~\ref{corollary:lym-V} is sharp for sparse paving matroids obtained from such Steiner systems.

\begin{lemma}\label{lemma:mn3_bound}
	For all~$n\ge 15$, $m(n,3) \le \left(\e(n-2)\right)^{\frac{1}{n-2}\binom{n}{3}} - 1$.
\end{lemma}

\begin{proof}
	A matroid of rank~3 on groundset~$[n]$ is determined by the triple~$(\mathcal{V}_0,\mathcal{V}_1,\mathcal{V}_2)$, where~$\mathcal{V}_k \subseteq \binom{[n]}{k+1}$. Writing~$v_k = |\mathcal{V}_k|$, it follows from Lemma~\ref{lemma:lym-V} that~$v_0\frac{1}{n} + v_1\frac{1}{n} + v_2\frac{3}{\binom{n}{2}} \le 1$, and hence
	\begin{equation*}
		m(n,3) \le \sum\left\{\binom{\binom{n}{1}}{v_0} \binom{\binom{n}{2}}{v_1} \binom{\binom{n}{3}}{v_2} : v_0\frac{1}{n} + v_1\frac{1}{n} + v_2\frac{3}{\binom{n}{2}} \le 1\right\}.
	\end{equation*}
	If~$v_0+v_1 = v$, then~$\sum\left\{\binom{\binom{n}{1}}{v_0} \binom{\binom{n}{2}}{v_1} : v_0 + v_1 = v\right\} = \binom{\binom{n+1}{2}}{v}$ and~$v_2 \le (n-v)\frac{n-1}{6}$. Splitting the sum over the possible values of~$v$, we obtain
	\begin{equation*}
		m(n,3) \le \sum_{v=0}^n \binom{\binom{n+1}{2}}{v} \sum_{v_2=0}^{\frac{(n-v)(n-1)}{6}} \binom{\binom{n}{3}}{v_2} = \sum_{v_2 = 0}^{\frac{n(n-1)}{6}} \binom{\binom{n}{3}}{v_2} T_n(v_2),
	\end{equation*}
	with
	\begin{equation*}
		T_n(v_2) = \sum_{v=0}^{\left\lfloor n - v_2\frac{6}{n-1}\right\rfloor} \binom{\binom{n+1}{2}}{v}.
	\end{equation*}
	Note that~$T_n(v_2) = 1$ whenever~$n - v_2 \frac{6}{n-1} < 1$, while if~$n - v_2 \frac{6}{n-1} \ge 1$ we have
	\begin{equation*}
		T_n(v_2) \le \left(\frac{\e \binom{n+1}{2}}{\left\lfloor n-v_2\frac{6}{n-1}\right\rfloor}\right)^{\left\lfloor n - v_2\frac{6}{n-1}\right\rfloor} \le \left(\e\binom{n+1}{2}\right)^{\frac{6}{n-1} \left(\frac{n(n-1)}{6} - v_2\right)}.
	\end{equation*}
	It follows that~$T_n(v_2) \le (n-2)^{\frac{n(n-1)}{6} - v_2}$ as soon as~$n\ge 15$, and a slightly more careful analysis shows that in that case even~$T_n(0) \le (n-2)^{\frac{n(n-1)}{6}} - 1$. By the binomial theorem, we obtain
	\begin{equation*}
		m(n,3) \le \sum_{v_2=0}^{\frac{n(n-1)}{6}} \binom{\binom{n}{3}}{v_2} (n-2)^{\frac{n(n-1)}{6} - v_2} - 1 \le (n-2)^{\frac{n(n-1)}{6}} \left(1+\frac{1}{n-2}\right)^{\binom{n}{3}} - 1,
	\end{equation*}
	and the lemma follows by bounding~$1 + t \le \e^t$.
\end{proof}

\thmmnrbound*

\begin{proof}
	We will use Lemma~\ref{lemma:blowup}(i) in combination with the bound on~$m(n,3)$ in Lemma~\ref{lemma:mn3_bound} to prove the result. We obtain
	\begin{equation*}
		\log (m(n,r) + 1) \le \frac{\binom{n}{r}}{\binom{n-r+3}{3}} \log (m(n-r+3,3)+1)
			\le \frac{1}{n-r+1}\binom{n}{r} \log(\e(n-r+1)),
	\end{equation*}
	if only~$n-r+3\ge 15$.
\end{proof}

\section{A lower bound on the number of paving matroids}\label{sec:constructions}

Now that we have established good upper and lower bounds on the logarithm on the number of paving and sparse paving matroids of a fixed rank~$r\ge 3$, to wit
\begin{equation*}
	\frac{1}{n-r+1} \binom{n}{r} \log\left(\e^{1-r} (n-r+1)(1+o(1))\right) \le \log s(n,r) \le \log p(n,r) \le \frac{1}{n-r+1} \binom{n}{r} \log\left(\e(n-r+1))\right),
\end{equation*}
it is natural to ask how big the gap between~$s(n,r)$ and~$p(n,r)$ can get. The following construction shows that~$p(n,3)$ is asymptotically larger than~$s(n,3)$.
\thmpvss*
\begin{proof}
	Start from a sparse paving matroid~$M$ with set of circuit-hyperplanes~$\mathcal{H}$. We will assume that~$M$ is not the uniform matroid, so~$\mathcal{H}\neq\emptyset$.

	Pick a circuit-hyperplane~$H$, and an element~$e\not\in H$, set~$H'=H\cup\{e\}$, and consider the set system
	\begin{equation*}
		\mathcal{H}' \defeq \mathcal{H} \cup \{H'\} \setminus \left\{X \in \mathcal{H} : |X\cap H'| \ge 2\right\}.
	\end{equation*}
	A moment's reflection reveals that~$\mathcal{H}'$ is the collection of dependent hyperplanes of a paving matroid that is not sparse. Each sparse paving matroid with~$k$ circuit-hyperplanes gives rise to~$k(n-3)$ paving matroids in this way. On the other hand, each paving matroid that is obtained in this way can arise from at most~$4n^3$ distinct sparse paving matroids, corresponding to the choice of~$e$ in the unique hyperplane of size~4, and at most one circuit-hyperplane for each of the three pairs~$\{e,x\}$ in that hyperplane.

	We claim that for some positive constant~$c$, all but a vanishing fraction of sparse paving matroids of rank~3 on~$[n]$ have~$k \ge cn^2$. This is an immediate consequence of the observation that the number of sparse paving matroids with at most~$cn^2$ circuit-hyperplanes is at most
	\begin{equation*}
		\sum_{i = 0}^{cn^2} \binom{\binom{n}{3}}{i} \le  \left(\frac{\e}{6c} n\right)^{cn^2} = o(s(n,3)),
	\end{equation*}
	provided~$c$ is sufficiently small. This proves the lemma, as
	\begin{equation*}
		\frac{p(n,3)}{s(n,3)} \ge 1 + (1-o(1))\frac{cn^2 (n-3)}{4n^3} \to 1 + \frac{c}{4} > 1\qquad\text{as $n\to\infty$.} \qedhere
	\end{equation*}
\end{proof}

Unfortunately, the argument does not generalise to arbitrary rank: in general it can be shown that an $(1-o(1))$-fraction of sparse paving matroids gives rise to~$\Omega\left(n^r\right)$ paving matroids, while the number of sparse paving matroids giving rise to the same paving matroid is~$O\left(n^{\binom{r}{2}}\right)$, and these bounds do not compare for~$r\ge 4$.

However, if we restrict our attention to sparse paving matroids with more underlying structure, we are able to obtain a more interesting comparison. In particular, if we restrict ourselves to Steiner systems, we obtain the following result (recall that~$D(n,r)$ denotes the number of~$S(r-1, r, n)$).
\begin{theorem}\label{thm:paving-vs-steiner}
	$p(n,r) \ge \left(1 + \left(\frac{n}{r} - \frac{n}{n-1}\right)\binom{n}{r-1}\right) D(r-1, r, n)$.
\end{theorem}

The theorem is proved by showing that each~$S(r-1, r, n)$ gives rise to~$1+\left(\frac{n}{r} - \frac{n}{n-1}\right)\binom{n}{r-1}$ paving matroids: one of them is sparse, and the others are not.

The non-sparse matroids are constructed analogously to the construction in Theorem~\ref{thm:pn3-vs-sn3}. Starting from a Steiner system~$S(r-1, r, n)$ with blocks~$\mathcal{H}$, pick any block~$H \in\mathcal{H}$, and any element~$e \in [n]$ that is not contained in~$H$. Write~$H' = H\cup\{e\}$, and consider the set system
\begin{equation*}
	\mathcal{H}' \defeq \mathcal{H} \cup \{ H'\} \setminus \left\{X \in \mathcal{H} : |X \cap H'| \ge r-1\right\}.
\end{equation*}
It is easily verified that~$\mathcal{H}'$ is the collection of hyperplanes of size~$\ge r$ of a paving matroid on~$[n]$ of rank~$r$. The theorem is now a straightforward consequence of the following lemma.
\begin{lemma}
	The Steiner system~$S(r-1,r,n)$, as well as the element~$e$ and the block~$H$ can be recovered from~$\mathcal{H}'$.
\end{lemma}
\begin{proof}
	As~$\mathcal{H}'$ contains a single set of cardinality~$r+1$, the set~$H + e$ can be recovered. Consider the set~$U$ of~$(r-1)$-sets of~$[n]$ that are not yet contained in any element of~$\mathcal{H}'$. The element~$e$ is contained in~$(r-2)\binom{r}{r-2}$ elements of~$U$, while every~$w \in H$ is contained in~$(r-2)\binom{r-1}{r-3}$ elements of~$U$. As~$\binom{r}{r-2} \neq \binom{r-1}{r-3}$, this allows us to recognise~$e$ (and hence~$H$) among the elements in~$H+e$.

	It remains to reconstruct the blocks that were deleted from~$\mathcal{H}$. For every~$X \in \binom{H}{r-2}$, there is a unique~$x \equiv x(X) \not \in H$ such that~$X + x \in U$, which must have come from a block~$X + x + v \in \mathcal{H}$. It follows that
	\begin{equation*}
		\mathcal{H} = \mathcal{H}' \cup \{H\} \cup \left\{X+x(X) + v : X \in \binom{H}{r-2}\right\},
	\end{equation*}
	which proves the lemma.
\end{proof}

Theorem~\ref{thm:paving-vs-steiner} raises the question how~$s(n,r)$ and~$D(n,r)$ compare. Note that~$s(n,r) \ge \frac{1}{n-r+1} \binom{n}{r} D(r-1, r, n)$, as deleting any block from an~$S(r-1, r, n)$ gives the set of circuit-hyperplanes of a unique sparse paving matroid.

\section{Closing remarks}\label{sec:closing}

\subsection{The number of paving matroids of fixed rank}

The best bound on~$\log s(n,r)$ was established in Theorem~\ref{thm:snr_entropy}, where is was shown that~$\log s(n,r) \le \frac{1}{n-r+1} \binom{n}{r} \log (n-r+2)$ using an entropy argument. Unfortunately, we are not able to use the same technique to prove a similar bound on~$\log p(n,r)$, for which the best bound remains~$\log p(n,r) \le \frac{1}{n-r+1} \binom{n}{r} \log(\e(n-r+1))$.

At this point, we would like to mention that improving the latter bound for rank-3 paving matroids immediately translates to improved upper bound for general fixed rank: If we assume that~$\log p(n,3) \le \frac{1}{n-2}\binom{n}{3} \log(c(n-2))$ for some constant~$c$, an application of Lemma~\ref{lemma:blowup}(ii) implies
\begin{equation}\label{eq:pnr_entropy_c}
	\log p(n,r) \le \frac{1}{n-r+1} \binom{n}{r} \log(c(n-r+1)).
\end{equation}
We could have applied the same trick to~$\log p(n,2)$ instead of~$\log p(n,3)$, but this gives an unwanted extra factor~2 in the estimate. We note that the constant~$c$ appearing in~\eqref{eq:pnr_entropy_c} does not depend on~$r$, while the corresponding factor in the best lower bound, to wit~$\e^{1-r}$, does.
Incidentally, Theorem~\ref{thm:p_upperbound} may be proved by first proving the special case~$r=3$ and then applying~\eqref{eq:pnr_entropy_c} with~$c=\e$, to obtain the general version.

To close this section, we mention that paving matroids of rank~3, or equivalently 2-partitions, are precisely partitions of the edges of~$K_n$ where each partition class forms a complete graph on at most~$n-1$ vertices.

\subsection{Comparing~$s(n,r)$, $p(n,r)$, and~$m(n,r)$ for constant~$r$}

We showed that~$\log s(n,r) \sim \log p(n,r) \sim \log m(n,r)$ for all fixed~$r\ge 3$. Unfortunately, our methods are not strong enough to decide whether~$s(n,r) \sim m(n,r)$ as well.

We hazard the guess that~$\frac{s(n,r)}{p(n,r)} \to 0$, while~$\frac{p(n,r)}{m(n,r)} \to 1$ for sufficiently large (but constant)~$r$, but emphasize that this problem is still open.

The evidence pointing to our first guess is stronger than that for the second guess. To start with, it is overwhelmingly true for~$r=2$ (in fact, $p(n,2)$ is roughly the square of~$s(n,2)$), while~$\frac{p(n,2)}{m(n,2)} \approx \frac{1}{\log n}$. Moreover, the constructions developed in the proofs of Theorem~\ref{thm:pn3-vs-sn3} and Theorem~\ref{thm:paving-vs-steiner} suggest that for~$r\ge 3$, $p(n,3)$ is much larger than~$s(n,3)$.

We are much less convinced about our second guess, which is based on two observations. First, the best bounds on~$p(n,r)$ and~$m(n,r)$ are the same, and both results can be obtained by a similar argument (by proving the corresponding statement for~$r=3$, followed by an application of Lemma~\ref{lemma:blowup}). A second observation that we make is that the uniform matroid has the largest number of erections among all matroids with the same rank and groundset. Writing~$\eta(n,r)$ for the average number of nontrivial erections over all non-uniform matroids on~$[n]$ of rank~$r$, we obtain
\begin{equation*}
	m(n,r) = p(n,r) + (m(n,r-1) - 1)\eta(n,r-1).
\end{equation*}
In the proof of Theorem~\ref{thm:fixrank}, we used that~$\eta(n,r-1) \le p(n,r)$. However, computer experiments suggest that the distribution of the number of nontrivial erections is very skew. It is therefore conceivable that~$m(n,r-1)\eta(n,r-1)$ is small compared to~$p(n,r)$, which would prove that~$\frac{p(n,r)}{m(n,r)} \to 1$. The results in Section~\ref{sec:erection_paving} lend some credibility to this argument.


\subsection{Counting antichains}

It was shown in Section~\ref{sec:antichains} that every matroid~$M$ can be reconstructed from its groundset, rank, and the antichain~$\mathcal{V}(M)$. In fact, the most important step in our best bound on~$m(n,r)$ in Theorem~\ref{thm:mnr_bound} relies on the LYM-type inequality bounding the cardinalities of the different levels of~$\mathcal{V}(M)$ that was obtained in Lemma~\ref{lemma:lym-V}.

We expect that these bounds can be improved by exploiting additional structure of the antichain~$\mathcal{V}(M)$. Some necessary structural results are listed in Lemma~\ref{lemma:V-properties}. We ask for an intrinsic (axiomatic) description of antichains that can appear as~$\mathcal{V}(M)$.

It is a consequence of Lemma~\ref{lemma:reconstruction-V} (in particular the remark immediately following it) that if~$M$ is a matroid such that the antichain~$\mathcal{V}(M)$ contains only sets of cardinality at least~$k$, then~$r(M) \ge k-1$, with equality if and only if~$M$ is the uniform matroid. This observation is potentially useful in proving the following two conjectures.

\begin{conjecture}[{\cite[Conjecture~1.10]{MayhewNewmanWelshWhittle2011}}]\label{conj:rank}
	Asymptotically almost all matroids satisfy
	\begin{equation*}
		\frac{n-1}{2} \le r(M) \le \frac{n+1}{2}.
	\end{equation*}
\end{conjecture}

\begin{conjecture}[{\cite[Conjecture~1.6]{MayhewNewmanWelshWhittle2011}}]\label{conj:paving}
	Asyptotically almost all matroids are paving.
\end{conjecture}

Suppose that it can be shown that for asymptotically almost all matroids~$M$, the corresponding antichain~$\mathcal{V}(M)$ contains only elements of cardinality at least~$\frac{n-1}{2}$.  It would follow that asymptotically almost all matroids satisfy~$r(M) \ge \frac{n-1}{2}$. By duality, it also follows that asymptotically almost all matroids have rank at most~$\frac{n+1}{2}$, and combining these two results proves Conjecture~\ref{conj:rank}.



Apart from Conjecture~\ref{conj:rank}, the assumption that for asymptotically almost all matroids, $\mathcal{V}(M)$ does not contain any elemens of cardinality less than~$\frac{n-1}{2}$ would also imply that
\begin{equation*}
	\lim_{n\to\infty} \frac{p(2n)}{m(2n)} = 1,
	\qquad\text{and}\qquad
	\liminf_{n\to\infty} \frac{p(2n+1)}{m(2n+1)} \ge \frac{1}{2},
\end{equation*}
which is just shy of Conjecture~\ref{conj:paving}.

The problem of counting antichains in the lattice of subsets is known as Dedekind's problem. It was solved asymptotically by Korshunov~\cite{Korshunov1981}, who obtained sharp estimates on the number of antichains based on the idea that most antichains are contained in the union of the three or four central levels (depending on the parity of~$n$). Korshunov's result is very precise at the cost of being very technical, and it is likely that results of comparable strength for maroids will be hard to obtain. (Korshunov's result, and many related results, are surveyed in~\cite{Korshunov2003}.)

We suggest the following weaker conjecture as something that may be easier to prove.
\begin{conjecture}
	Let~$k$ be fixed. Asymptotically almost all antichains~$\mathcal{V}(M)$ do not contain a subset of cardinality at most~$k$.
\end{conjecture}

If this conjecture holds, it would imply that asymptotically almost all matroids do not contain a circuit of cardinality at most~$k$.


\section{Acknowledgement}

We would like to thank Dillon Mayhew for stimulating discussions on matroids and antichains, and Henry Crapo for his comments on an earlier version of this paper.


\bibliography{bib}
\bibliographystyle{alpha}

\end{document}